\documentclass{amsart}

\usepackage{amscd}

\usepackage[utf8]{inputenc}
\usepackage{amssymb}
\usepackage{amsfonts}
\usepackage{amsmath}
\usepackage{amsthm}
\usepackage{enumerate}
\usepackage{mathrsfs}
\usepackage[pdftex]{graphicx}
\usepackage{wrapfig}
\usepackage{comment}
\usepackage{graphicx}
\usepackage{mathtools}

\usepackage{hyperref}
\hypersetup{
    colorlinks=true,
    linkcolor=blue,
    filecolor=magenta,      
    urlcolor=cyan,
    citecolor=magenta
}
\usepackage{cleveref}

\makeatletter
\newtheorem*{rep@thm}{\rep@title}
\newcommand{\newreptheorem}[2]{%
\newenvironment{rep#1}[1]{%
 \def\rep@title{#2 \ref{##1}}%
 \begin{rep@thm}}%
 {\end{rep@thm}}}
\makeatother

\newtheorem{thm}{Theorem}[section]
\newtheorem{cor}[thm]{Corollary}
\newtheorem{conj}[thm]{Conjecture}
\newtheorem{prop}[thm]{Proposition}

\newtheorem{lem}[thm]{Lemma}

\newreptheorem{thm}{Theorem}
\newreptheorem{lem}{Lemma}

\theoremstyle{definition}
\newtheorem*{defn}{Definition}
\newtheorem*{conv}{Convention}

\newtheorem{exmp}[thm]{Example}

\theoremstyle{remark}

\newcommand{\N}{\mathbb N}

\newcommand{\Q}{\mathbb Q}
\newcommand{\R}{\mathbb R}
\newcommand{\C}{\mathbb C}

\newcommand{\la}{\langle}
\newcommand{\ra}{\rangle}

\let\emptyset\varnothing

\let\epsilon\varepsilon

\let\phi\varphi
\let\bdy\partial
\newcommand*{\mb}[1]{\ensuremath{\mathbb{#1}}}

\newcommand*{\mf}[1]{\ensuremath{\mathfrak{#1}}}

\newcommand*{\bs}{\backslash}

\newcommand{\laa}{\la\kern-1.2ex~\la}
\newcommand{\raa}{\ra\kern-1.2ex~\ra}

\DeclareMathOperator{\SL}{SL} 
\DeclareMathOperator{\ad}{ad}

\title{Algebraic Varieties and Automorphic Functions}
\author[Sebastian Eterovi\'c]{Sebastian Eterovi\'c}
\address{Kurt G\"odel Research Center, Universit\"at Wien, 1090 Wien, Austria}
\email{sebastian.eterovic@univie.ac.at}

\author[Roy Zhao]{Roy Zhao}
\address{Department of Mathematics, California Institute of Technology, Pasadena, CA, USA.}
\email{rhzhao@caltech.edu}

\thanks{Both authors were supported by NSF RTG grant DMS-1646385. SE was also supported by EPSRC fellowship EP/T018461/1}



\begin{document}

\begin{abstract}
Let $(G, X)$ be a Shimura datum, let $\Omega$ be a connected component of $X$, let $\Gamma$ be a congruence subgroup of $G(\Q)^{+}$, and consider the quotient map $q: \Omega \to S:=\Gamma \bs \Omega$. Consider the Harish-Chandra embedding $\Omega\subset\mb{C}^{N}$, where $N=\dim X$. We prove two results that give geometric conditions which, if satisfied by an algebraic variety $V \subset \C^{N} \times S$, ensure that there is a Zariski dense subset of $V$ of points of the form $(x,q(x))$. 
\end{abstract}

\maketitle

\section{Introduction}
Let $(G, X)$ be a Shimura datum with $G$ reductive over $\Q$ and $K$ a compact open subgroup of $G(\mb{A}_f)$, where $\mb{A}_f$ will denote the ring of finite rational ad\`eles.
Let $G^{\ad}$ denote the adjoint group which is the quotient of $G$ by its center and let $G(\R)_+$ denote the elements of $G(\R)$ which map to the identity component of $G^{\ad}(\R)$.
Let $G(\Q)_+ \coloneqq G(\Q) \cap G(\R)_+$.
Let $\Omega$ be a connected component of $X$ and let $\Gamma \coloneqq G(\Q)_+ \cap K$, then the quotient
\[S \coloneqq  \Gamma \bs \Omega\]
has the structure of an algebraic variety and the quotient map $q \colon \Omega \to S$ is analytic and transcendental.
Let $E_q\coloneqq\left\{ (x,q(x))\in\Omega \times S\right\}$ denote the graph of $q$.
Definitions and properties of Shimura varieties can be found in \cite[\S 5]{Mil05}. 
The $G(\R)$-conjugacy class $X$ is not an algebraic set, but each connected component $\Omega \subset X$ can be given the structure of a Hermitian symmetric domain and has a realization as an open immersion into its dual compact Hermitian symmetric manifold $\widehat{\Omega}$ of compact type through the Borel embedding. 
By the Harish-Chandra embedding theorem, we can realize $\Omega$ as a bounded domain on some Zariski open subset $\mathbb{W}$ of $\widehat{\Omega}$, with $\mathbb{W}$ biregular to $\C^N$, where $N$ is the dimension of $S$.
We will henceforth identify $\mathbb{W}$ with $\C^N$ and write $\Omega \subset \C^N$ for the Harish-Chandra realization.

In this article we study the following problem, sometimes known as the \emph{existential closedness problem} (ECP) for $q$: find a minimal set of geometric conditions that an algebraic variety $V\subset \C^{N}\times S$ has to satisfy in order for $V$ to have a Zariski dense set of points in the graph of $q$. 
We will present results on this question in two families of varieties. 
By the Ax--Schanuel Theorem for Shimura varieties \cite{Mok19}, when $\dim V=N$, the dimension of each irreducible component $U$ of $V\cap E_q$ is zero unless the natural projection of $U\subset (\Omega\times S)$ to $\Omega$ lies in a proper \emph{weakly special subvariety} $Y\subsetneq\Omega$ (definition in \S5.1). 

Our first main result gives geometric conditions under which the intersection $V\cap E_{q}$ is not only non-empty, but is as large as it could be (i.e.~Zariski dense in $V$). 

\begin{thm} \label{thm:ZariskiDenseIntersection}
	Let $V \subset \C^{N} \times S$ be an irreducible algebraic variety and let $\pi\colon \C^{N} \times S \to \C^{N}$ denote the projection. 
 If $\pi(V)$ is Zariski dense in $\C^{N}$, then $\pi(V \cap E_q)$ is Zariski dense in $\C^N$, and $V \cap E_q$ is Zariski dense in $V$.
\end{thm}

Although we will deal mostly with pure Shimura varieties in this article, we are also able to combine Theorem \ref{thm:ZariskiDenseIntersection} with the an analogous result for the exponential maps of abelian varieties proven in \cite{aslanyan-kirby-mantova}. 
We are able to obtain a result for those mixed Shimura varieties which are families of abelian varieties, see Corollary \ref{cor:mixed}.

\subsection{Broad and free varieties}
For our second main result, instead of relying on a dominant projection to the first set of coordinates, we look at varieties of the form $L \times W$, where $L$ is a ``subvariety'' of $\Omega$ and $W$ is a subvariety of $S$.
By an \emph{irreducible subvariety of} $\Omega \subset\mathbb{C}^{N}$, we mean a complex-analytically irreducible component of the intersection of $\Omega$ with an algebraic subvariety of $\mathbb{C}^{N}$. 
In order to state our second main result, we need to set-up some definitions. 
We follow the conventions of \cite{Ull11} for defining weakly special subvarieties and \cite{Ull18} for totally geodesic subvarieties (we will recall some of these aspects in \textsection\ref{subsec:totallygeodesic}).

\begin{defn}\label{defn:decomposition of Shimura Variety}
    For every decomposition of $G^{\ad} = G_1 \times G_2$ into a product of two normal $\Q$-subgroups (possibly trivial), we get a splitting of the Hermitian symmetric domain $\Omega = \Omega_1 \times \Omega_2$.
    Let $g_i \coloneqq G \to G^{\ad} \to G_i$ denote the projections.
    Given a compact open subgroup $K$ of $G(\mb{A}_f)$, let $K_i \coloneqq g_i(K)$, which is a compact open subgroup of $G_i(\mb{A}_f)$. 
    Thus, by setting $\Gamma \coloneqq K\cap G(\mathbb{Q})_+$ and $\Gamma_i\coloneqq K_i\cap G_i(\mathbb{Q})_+$, we obtain projections $S \coloneqq \Gamma\backslash\Omega\to\Gamma_i\backslash\Omega_i \eqqcolon S_i$, which then induce projections $p_i \colon \Omega \times S \to \Omega_i \times S_i$.

    We say that a variety $V\subset \mb{C}^{N}\times S$ is \emph{broad} if $\dim p_i(V) \ge \dim S_i$ for every such splitting of $G^{\ad}$. 

    A subvariety $W \subset S$ is said to be \emph{free} if there does not exist a proper weakly special subvariety $S' \subset S$ such that $W \subset S'$.
    A subvariety $Z\subset\mathbb{C}^{N}$ is \emph{free} if $\overline{q(Z\cap\Omega)}^{\mathrm{Zar}} = S$ (cf \cite[Theorem 1.2]{Ull18}). 
    
    Furthermore, we say that an algebraic variety $V\subseteq\mb{C}^{N}\times S$ is \emph{free} if the projections of $V$ to $\mb{C}^{N}$ and $S$ are both free. 
\end{defn}

We can now state our second main result. 

\begin{thm}
\label{thm:ProductDenseImage}
	Let $L \subset \Omega$ be a totally geodesic subvariety and let $W \subset S$ be an algebraic variety such that $L\times W$ is a free and broad variety. 
 Then, the intersection $W \cap q(L)$ is Euclidean dense inside $W$. \end{thm}

Given the conjecture we formulate in the next subsection, we further expect $(L\times W)\cap E_q$ to be Zariski dense in $L\times W$.

\subsection{Background on ECP problems}
Theorems \ref{thm:ZariskiDenseIntersection} and \ref{thm:ProductDenseImage} follow an increasing body of work studying ECP for different maps arising in arithmetic geometry. 
Originally formulated in Zilber's work on pseudoexponentiation \cite{Zil05}, ECP began as a question about the complex exponential function ($\exp$). 
It has since been adapted to other functions. 
Results analogous to Theorem \ref{thm:ZariskiDenseIntersection} are given in \cite{Mar06}, \cite{Man16}, \cite{brownawell-masser}, \cite{daquino-exp}, \cite{aslanyan-kirby-mantova} for $\exp$, in \cite{aslanyan-kirby-mantova} for the exponential of algebraic tori, of any abelian variety, and of split semi-abelian varieties; in \cite{Ete21} for the $j$-function, and in \cite{eterovic-padgett} for the $\Gamma$ function. 
On the other hand, Gallinaro considered varieties like the ones in Theorem \ref{thm:ProductDenseImage} in the context of the $j$-function \cite{Gal21}, the exponentials of abelian varieties \cite{gallinaro-abelian}, and the complex exponential function \cite{gallinaro-exp}.  
This last result improves earlier results of Zilber \cite{zilber2015theoryexponentialsums} and was then used in \cite{gallinaro-kirby} to prove quasiminimality of certain model-theoretic structures.

There are conjectural answers to ECP for all the functions mentioned above, and we now extend naturally those conjectures to the setting of Shimura varieties.
\begin{conj}
\label{conj:ecp}
Let $V \subset \C^{N} \times S$ be an irreducible algebraic variety which is broad and free. Then $V \cap E_q$ is Zariski dense in $V$.
\end{conj}
The condition in Theorem \ref{thm:ZariskiDenseIntersection} that the projection $\pi(V)$ be Zariski dense in $\C^{N}$ is stronger than broadness.

\subsection{Structure of the article}

In \textsection \ref{sec:HSD} we go over some basic facts of Hermitian symmetric domains and their various embeddings as well as metrics defined on them. 
Then in \textsection \ref{sec:boundary} we cover the structure of the boundary of $\Omega$ and focus our discussion on the Shilov boundary. 
We prove that the $\Gamma$-orbit of any point in $\Omega$ contains a Zariski dense subset of the Shilov boundary in its closure. 
With this result, we are ready to present the proof of Theorem \ref{thm:ZariskiDenseIntersection} in \textsection \ref{sec:proofDenseImage}. 
We define weakly special and totally geodesic subvarieties of $\Omega$, the notions of broad and free, and present some consequences from Ratner theory in \textsection \ref{sec:WeaklySpecial} which will allow us to prove Theorem \ref{thm:ProductDenseImage} in \textsection \ref{sec:proofProduct}. 

\subsection*{Acknowledgements} We wish to thank the anonymous referees for their valuable comments that helped improve the presentation of this article.

\section{Shimura Varieties and Hermitian Symmetric Domains}\label{sec:HSD}
\subsection{Shimura Varieties}
We briefly review the theory of Shimura varieties. 
More detailed information can be found in \cite{Mil05}.

A Shimura datum is a pair $(G, X)$ where $G$ is a reductive algebraic group over $\Q$ and $X$ is a $G(\R)$-conjugacy class of homomorphisms $h: \mb{S} \to G_\R$ from the Deligne torus $\mb{S} := \text{Res}_{\C/\R}\mb{G}_m$ such that one (and hence all) $h \in X$ satisfies the following three conditions:

\begin{enumerate}[(1)]
    \item The Hodge structure on the Lie algebra of $G_\R$ given by $\text{Ad} \circ h$ is of type $\{(-1, 1), (0, 0), (1, -1)\}$,
    \item the adjoint action $\text{Ad} h(i)$ is a Cartan involution on the adjoint group $G^{\ad}_\R$,
    \item the adjoint group $G^{\ad}$ has no $\Q$-factor on which the projection of $h$ is trivial.
\end{enumerate}

When these three conditions are satisfied, then $X$ has the structure of a disjoint union of Hermitian symmetric domains. 
Let $\mb{A}_f$ denote the ring of finite ad\`eles of $\Q$ and let $K$ be a compact open subgroup of $G(\mb{A}_f)$. 
The Shimura variety attached to the triple $(G, X, K)$ is the double quotient
\[\text{Sh}_K(G, X) := G(\Q) \bs X \times \left(G(\mb{A}_f)/K\right),\]
where $G(\Q)$ acts on both $X$ by conjugation and $G(\mb{A}_f)/K$ by left multiplication. 
When $K$ is sufficiently small, this double coset space has the structure of a smooth quasi-projective variety over $\C$, and moreover has a canonical model over a number field.
 
Let $\Omega$ denote a connected component of $X$ and let $\mathcal{C}$ be a set of representatives for the double coset $G(\Q)_+ \bs G(\mb{A}_f) / K$. 
Then our Shimura variety can be described as a disjoint union of quotients of Hermitian symmetric domains
\[\text{Sh}_K(G, X) \simeq \sqcup_{g \in \mathcal{C}} \Gamma_g \bs \Omega,\]
where $\Gamma_g = gKg^{-1} \cap G(\Q)_+$ is an arithmetic subgroup of $G(\R)$.

We limit ourselves to looking at the quotient map associated with exactly one of these connected components. 
Fix a connected component $\Omega$ of $X$ and fix $\Gamma \subset G(\Q)_+$ an arithmetic subgroup of $G(\R)$ associated with a $g \in \mathcal{C}$ and compact open subgroup $K \subset G(\mb{A}_f)$, and let $S := \Gamma \bs \Omega$ denote the connected component of $\text{Sh}_K(G, X)$ corresponding to $\Omega$ and $\Gamma$, and $q: \Omega \to S$ the quotient map from modding out by the left action of $\Gamma$.

\subsection{Borel and Harish-Chandra embeddings}
The quotient $S$ is a connected component of the full Shimura variety and thus has the structure of a smooth algebraic variety. 
In order to discuss varieties of $\Omega \times S$, we need to give the $G(\R)^+$-conjugacy class of homomorphisms $\Omega$ algebraic structure.

Every point $h \in \Omega$ is a homomorphism $h: \mb{S} \to G_\R$, and to each we can attach a cocharacter of $G_\C$ by taking
\[\mu_h: \mb{G}_{m, \C} \to \mb{S}_\C \to G_\C\]
where the first map takes $z \mapsto (z, 1) \in \mb{S}_\C = \C^\times \times \C^\times$, and the second map is induced by $h$. 
This cocharacter gives rise to a filtration $\text{Filt}(\mu)$ of $\textbf{Rep}_\C(G_\C)$. 
This functor gives a decreasing filtration $F^\bullet$ on $V$ for each representation $(V, \xi)$ of $G_\C$ where $F^pV = \oplus_{i \ge p}V^i$ and $V^i = \{v \in V : \mu(z)v = z^{-i}v\}$. 
The compact dual Hermitian space $\widehat{\Omega}$ to $\Omega$ is the $G_\C$-conjugacy class of filtrations of $\textbf{Rep}_\C(G_\C)$ that contains $\text{Filt}(\mu)$.

Let $P_\mu \subset G_\C$ be the subgroup fixing the filtration $\text{Filt}(\mu)$ of $\textbf{Rep}_\C(G_\C)$. 
It is a parabolic subgroup of $G_\C$ and so the bijection $G(\C)/P_\mu(\C) \to \widehat{\Omega}$ endows $\widehat{\Omega}$ with the structure of a smooth projective complex variety. 
There is a natural embedding $\Omega \hookrightarrow \widehat{\Omega}$ given by sending a homomorphism $h \in \Omega$ to the filtration associated with it $\text{Filt}(\mu_h)$. 
Fix a base point $o \in \Omega$ and let $K_o \subset G(\R)^+$ be the subgroup fixing the homomorphism $o$ and let $P_o := P_{\mu_o}$ be the parabolic subgroup associated with $\mu_o$. 
Then we can identify $\Omega$ and $\widehat{\Omega}$ with the coset spaces $G(\R)^+/K_o$ and $G(\C)/P_o(\C)$ respectively. 
Additionally $K_o = G(\R)^+ \cap P_o(\C)$, and the Borel embedding can be seen as the natural map
\[\Omega = G(\R)^+/K_o \hookrightarrow G(\C)/P_o(\C) = \widehat{\Omega}.\]
In our situation, we are more interested in the Harish-Chandra embedding. 
Let $\mf{p}^+$ denote the holomorphic tangent space of $\widehat{\Omega}$ at $o$. 
It is proven in \cite[Theorem 2.5.6]{Str06} that $\Omega$ can be realized in $\mf{p}^+ \simeq \C^N$ as a bounded symmetric domain. 
Moreover, this tangent space can be embedded $\mf{p}^+ \to \widehat{\Omega}$ as a dense open subset in the compact dual of $\Omega$. 
In this way, we will view $\Omega$ as a bounded symmetric domain inside $\C^N$ for $N = \dim \mf{p}^+ = \dim \Omega$, and the algebraic structure on $\Omega$ will be induced from $\C^N$. 
In fact, this gives a bijection between Hermitian symmetric domains and bounded symmetric domains. 
A more detailed discussion can be found in \cite[Chapter VIII]{Hel79}

\begin{exmp}
	If $(G, \Omega) = (\mathrm{SL}_2, \mb{H})$, then for any choice of base point $o \in \Omega$, the stabilizer $P_o$ is a subgroup of $\mathrm{SL}_2(\C)$ conjugate to $B := \left\{\begin{pmatrix}
		a & b \\ 0 & a^{-1}
	\end{pmatrix}: a \in \C^\times, b \in \C\right\}$ and so $\widehat{\Omega} \simeq \mb{P}^1(\C)$ and $\mf{p}^+ \simeq \C$. The Harish-Chandra embedding realizes $\Omega$ as the unit disk $\mb{D} \subset \C$.
\end{exmp}

We define an algebraic subvariety of $\Omega \subset \C^N$, which is an analytic but not an algebraic space, as the restriction to $\Omega$ of an algebraic variety of $\C^N$.
\begin{defn}
	A subset $Y \subset \Omega$ is an \emph{irreducible algebraic variety of $\Omega$} if there exists an algebraic variety $Z \subset \C^N$ such that $Y$ is an irreducible component of the complex analytic subvariety $\Omega \cap Z\subset\Omega$. 
 A subset $Y \subset \Omega$ is called an \emph{algebraic subvariety} if $Y$ is a complex analytic subvariety of $\Omega$ with finitely many irreducible branches and each of the irreducible components of $Y$ is an irreducible algebraic subvariety of $\Omega$.  
\end{defn}

\begin{conv}
    Although the arithmetic group $\Gamma$ may not be torsion-free, we can find a finite index subgroup $\Gamma'\leq \Gamma$ so that $\Gamma'$ is torsion-free. 
    Set $S':=\Gamma'\backslash\Omega$ and $S:=\Gamma\backslash\Omega$, and denote the corresponding maps $q':\Omega\to S'$ and $q:\Omega\to S$. 
    We then get a natural finite algebraic map $S'\to S$, which in turn induces a finite algebraic map $\xi:\C^N\times S'\to \C^N\times S$.

    Given an algebraic variety $V\subset \C^N\times S$ and a component $V'$ of $\xi^{-1}(V)$, we get that $\xi\left(V'\cap E_{q'}\right)\subset V\cap E_q$. 
    Therefore, when proving Theorems \ref{thm:ZariskiDenseIntersection} and \ref{thm:ProductDenseImage} we may and will assume that $\Gamma$ is torsion-free. 
    This is a common and usually harmless convention to make when working with Shimura varieties, and for us it will be used explicitly in the proof of Proposition \ref{prop:InfinitelyManySolutions}. 
\end{conv}

\subsection{Metrics on domains}
We will be interested in a comparison between different notions of distance on $\Omega$ and so we briefly recall different metrics that can be defined on a general open domain $D \subset \C^n$. 
In addition to the Euclidean metric, there also exists the Bergman and Carath\'eodory metrics on a domain $D \subset \C^n$. We first give a description of the Bergman metric, which can also be found in more detail in \cite[Chapter VIII]{Hel79}.

For a domain $D \subset \C^n$, let $H(D)$ denote the space of holomorphic functions $f: D \to \C$ and let
\[H^2(D) := \left\{f \in H(D) : \int_D |f(z)|^2dz < \infty\right\}\]
denote the subspace of holomorphic functions with bounded $L^2$ norm. 
The space $H^2(D)$ is a Hilbert space for the $L^2$ norm $\la f, f\ra = \int_D |f(z)|^2dz$. 
For each fixed $w \in D$, the Riesz representation theorem implies the linear functional $f \mapsto f(w)$ is representable by some function $K(\cdot, w) \in H^2(D)$, so that $\la f, K(\cdot, w) \ra = f(w)$ for all $f \in H^2(D)$. 
The function $K(z, w)$ is defined on $D \times D$, it is holomorphic in $z$ and anti-holomorphic in $w$, satisfies $K(z, w) = \overline{K(w, z)}$, and is called the Bergman kernel of $D$.

\begin{defn}
	For each $z \in D$ and tangent vector $\textbf{v} = \sum_{i = 1}^n a_i \frac{\bdy}{\bdy z_i} \in T_z(D)$, the \emph{Bergman metric} is
	\[B_D(z; \textbf{v})^2 := \sum_{i, j = 1}^n a_i\overline{a_j} \frac{\bdy^2}{\bdy z_i \bdy \overline{z_j}} \log K(z, z).\]
\end{defn}

From this metric, comes the notion of Bergman distance.

\begin{defn}
	For each pair of points $z, w \in D$, the \textit{Bergman distance} is
	\[b_D(z, w) := \inf_\gamma \int_0^1 B_D(\gamma(t); \gamma'(t))dt,\]
	where the infimum is taken over all piecewise $C^1$ curves $\gamma: [0, 1] \to D$ with $\gamma(0) = z$ and $\gamma(1) = w$.
\end{defn}

\begin{exmp}
	If $D = \mb{H} \subset \C$ is the upper half plane, then the Bergman metric corresponds to the Poincar\'e metric $B_D(z)^2 = \frac{dzd\bar{z}}{y^2}$.
	
	If $D = \mb{D} \subset \C$ is the unit disk, the Bergman metric is given by $B_D(z) = \frac{ds}{1 - |z|^2}$, where $ds$ denotes the Euclidean metric.
\end{exmp}

An important property of the Bergman distance is that it is invariant under holomorphic automorphisms. 
For Shimura varieties, on the Hermitian symmetric domain $\Omega$ each $g \in G(\R)_+$ acts as a holomorphic automorphism of $\Omega$ \cite{Mil05}. 
Thus, for each $g \in G(\R)_+$ and $z, w \in \Omega$, the Bergman distance $b_\Omega(z, w) = b_\Omega(gz, gw)$ is invariant under the action of $g$.

The Carath\'eodory--Reiffen metric introduced by Reiffen \cite{Rei63} is another biholomorphic metric that can be defined on bounded domains. 
While we are primarily interested in the Bergman metric, the Carath\'eodory metric will help bound the Bergman metric and is easier to work with.

Let
\[H^\infty(D) := \left\{f \in H(D) : \sup_{z \in D} |f(z)| < \infty \right\}\]
denote the subspace of holomorphic functions with bounded $L^\infty$ norm.
\begin{defn}
	For each $z \in D$ and tangent vector $\textbf{v} = \sum_{i = 1}^n a_i \frac{\bdy}{\bdy z_i} \in T_z(D)$, the \emph{Carath\'eodory metric} is
	\[C_D(z; \textbf{v}) := \sup \left\{\textbf{v}f : f \in H^\infty(D), f(z) = 0, \|f\|_\infty \le 1\right\}.\]
\end{defn}

As with the Bergman metric, we define the Carath\'eodory distance between two points using the Carath\'eodory metric. Note that this is different from the definition of the Carath\'eodory distance given in \cite[\S4]{Kra08}, which defines the Carath\'eodory distance in terms of the Poincar\'e distance. However, we will only need results about the Carath\'eodory metric and length, and not use results on the Carath\'eodory distance of \cite{Kra08}.

\begin{defn}
	For each pair of points $z, w \in D$, the \textit{Carath\'eodory distance} is
	\[c_D(z, w) := \inf_\gamma \int_0^1 C_D(\gamma(t); \gamma'(t))dt,\]
	where the infimum is taken over all piecewise $C^1$ curves $\gamma: [0, 1] \to D$ with $\gamma(0) = z$ and $\gamma(1) = w$.
\end{defn}

\begin{exmp}
	If $D = \mb{D}$, then the Carath\'eodory metric coincides with the Bergman metric and Poincar\'e metric.
\end{exmp}

There is a simple relation between the Carath\'eodory and Bergman metrics, in that the Bergman metric is at least the Carath\'eodory metric.

\begin{thm}[\cite{Hah78}]
    In any domain $D \subset \C^n$, for every point $z \in D$ and tangent vector $\textbf{v} \in T_z(D)$, we have $C_D(z; \textbf{v}) \le B_D(z; \textbf{v})$.
\end{thm}

\section{Boundaries of Hermitian Symmetric Domains} \label{sec:boundary}
\subsection{Shilov boundary}
A fundamental result from complex analysis says that a bounded domain $D \subset \C$ satisfies the maximum modulus principle: a holomorphic function on $\overline{D}$ achieves its maximum modulus on the boundary of $D$. 
In higher dimensions though, we can sometimes state a stronger result in that the maximum modulus must occur on a proper closed subset of the boundary of $D$.
\begin{defn}
	Let $D \subset \C^n$ be a bounded domain.
 The \emph{Shilov boundary} of $D$ is the smallest closed subset $\sigma(D) \subset \bdy D$ that satisfies the maximum modulus principle. 
 That is, for any holomorphic function defined in an open neighborhood of the closure $\overline{D}$, the function $|f(z)|$ achieves a maximum at some point $z \in \sigma(D)$.
\end{defn}

\begin{exmp}
	If $D = \mb{D} \times \mb{D} \subset \C^2$, then the Shilov boundary is $S^1 \times S^1$. 
 It satisfies the maximum modulus principle by applying the single variable maximum modulus principle in the first coordinate and then the second coordinate. 
 It is a proper closed subset of $\bdy D = (\mb{D} \times S^1) \sqcup (S^1 \times \mb{D}) \sqcup (S^1 \times S^1)$.
\end{exmp}

\begin{prop}\label{prop:ShilovDense}
	Let $D \subset \C^n$ be a bounded domain. Then $\sigma(D)$ is Zariski dense in $\C^n$.
\end{prop}
\begin{proof}
	Suppose for contradiction that $\sigma(D) \subset Z(f)$ was in the zero set of some non-zero polynomial $f$. 
 Then $f(z) = 0$ for all $z \in \sigma(D)$.
 By definition, the function $|f(z)|$ restricted to $\overline{D}$ achieves a maximum at some $z \in \sigma(D)$, so the function must vanish on all of $\overline{D}$. 
 The function $f: \C^n \to \C$ vanishes on an open set in $\C^n$ and hence on all of $\C^n$, so $f = 0$, contradicting our assumption on $f$. 
 Therefore $\sigma(D)$ does not lie in any proper variety of $\C^n$ and is Zariski dense.
\end{proof}

In the case of $D$ being a Hermitian symmetric domain, the boundary has a well-studied decomposition. The boundary inherits an action of $\text{Hol}(D)$. Writing $G$ for the identity component of $\text{Hol}(D)$, after partitioning the boundary of $D$ into disjoint $G$-orbits, each $G$-orbit is either a disjoint union of positive-dimensional boundary components, which are Hermitian symmetric domain of smaller dimension, or it is the unique closed orbit \cite[Part I.5]{Wol72}. This unique closed orbit is precisely the Shilov boundary of $D$.

\begin{exmp}
	If $D = \mb{D} \times \mb{D} \subset \C^2$, then the boundary of $D$ can be split up as $\bdy D = (\mb{D} \times S^1) \sqcup (S^1 \times \mb{D}) \sqcup (S^1 \times S^1)$. 
 The first two components consist of disjoint unions of Hermitian symmetric domains of smaller dimension, and the unique closed orbit is the Shilov boundary $S^1 \times S^1$.
\end{exmp}

\subsection{Behavior near Shilov boundary}
In the case of the upper half-plane, as the imaginary part of a point $z \in \mb{H}$ tends to zero, a hyperbolic ball of fixed radius centered at $z$ converges uniformly in the Euclidean metric to a real point. 
In the general case of a Hermitian symmetric domain of higher dimension though, this is not necessarily true. 
While it still holds that as we take the center $z \in D$ of a hyperbolic ball of fixed radius to a point on the boundary $\bar{z} \in \bdy D$, the Euclidean volume of the ball will go to $0$ and all the points in the ball will converge to the boundary \cite{Ber88}, but the points in the ball may not all converge to the same point on the boundary. 
For instance, taking the center $(z, w) \in D = \mb{D} \times \mb{D}$ of a hyperbolic ball to a point on the boundary $(\bar{z}, w) \in S^1 \times \mb{D} \subset \bdy D$ by keeping the second coordinate fixed and only changing the first coordinate makes the points in the ball converge to a hyperbolic ball $\{\bar{z}\} \times B \subset S^1 \times \mb{D}$ in a smaller Hermitian domain. 
However, if the center converges to a point on the Shilov boundary, the hyperbolic ball centered there will converge in the Euclidean metric to the same point on the Shilov boundary.

\begin{prop}\label{prop:LimitBehavior}
	Let $\Omega \subset \C^N$ be the bounded realization of a Hermitian symmetric domain. For any fixed real number $r \ge 0$ and $z' \in \sigma(D)$,
	\[\lim_{z \to z', z \in \Omega} \sup \{\|w - z\| : w \in \Omega, b_\Omega(z, w) < r\} = 0\]
\end{prop}

\begin{proof}
	The Shilov boundary corresponds to the set of points in $\overline{\Omega}$ of maximal Euclidean distance from the origin \cite[Theorem 6.5]{Loo77}. 
 Let the radius of $\Omega$ to be $R = \sup_{s \in \Omega} \|s\|$. Let $B_R \subset \C^N$ denote the ball of radius $R$ centered at the origin. 
 From the definition of the Carath\'eodory metric, we must have $C_{B_R}(z; \textbf{v}) \le C_\Omega(z; \textbf{v})$ for any $z \in \Omega \subset B_R$ and $\textbf{v} \in T_z(\Omega) = T_z(B_R)$.
 Moreover, the Carath\'eodory metric is less than the Bergman metric and hence
	\[C_{B_R}(z; \textbf{v}) \le C_\Omega(z; \textbf{v}) \le B_\Omega(z; \textbf{v}).\]
	For each $w \in \Omega$ with $b_\Omega(z, w) < r$, integrating along paths gives $c_{B_R}(z, w) \le b_\Omega(z, w) < r$. 
 The point $z' \in \sigma(\Omega)$ is on the Shilov boundary with $\|z'\| = R$ and so as $z \to z'$ we have $z \to \bdy B_R$ as well.
 The inequality $c_{B_R}(z, w) \le r$ means that $w$ is in a ball of fixed Carath\'eodory radius centered at $z$ and as $z \to \bdy B_R$, and hence $\|z - w\| \to 0$ (see paragraph after proof of \cite[Proposition 17]{Kra08}).
\end{proof}

\subsection{Limit sets of arithmetic subgroups}
The complex points of a connected component of a Shimura variety can be described as the quotient of a Hermitian symmetric domain $\Omega$ by the action of an arithmetic subgroup $\Gamma \subset G(\Q)_+$. 
We would like to know the behavior of $\overline{\Gamma x}$, the Euclidean closure of the $\Gamma$-orbit of some point $x \in \Omega$.
The group $\Gamma$ acts discretely on $\Omega$ and hence any limit points of $\Gamma x$ must lie on the boundary $\bdy \Omega$. 
In the case of the upper half-plane $\Omega = \mb{H}$, the arithmetic subgroup $\Gamma$ is a Fuchsian group of the first type and the limit set of the orbit of any point $\overline{\Gamma x}$ contains the boundary $\bdy \mb{H}$ and hence is equal to $\bdy \mb{H}$, which is the real line and the point at infinity. 
Again, by looking at the case of $\mb{D} \times \mb{D}$, this result does not hold in higher dimensions but as before, the Shilov boundary provides a suitable substitute.

\begin{prop}\label{prop:ShilovBoundaryLimit}
    Suppose $\Omega$ be a Hermitian symmetric domain acted on by an arithmetic subgroup $\Gamma \subset G(\Q)_+$.
    For any point $x \in \Omega$, the limit set of the $\Gamma$-orbit of $x$ contains $\sigma(\Omega) \subset \overline{\Gamma x}$.
\end{prop}
\begin{proof}
    By \cite[\S 3.6]{Ben97}, we know that there exists a point $y \in \sigma(\Omega) \cap \overline{\Gamma x}$.
    By \cite[Theorem 3.6]{KW65}, the action of $G(\R)_+$ extends to the boundary on which the Shilov boundary is the unique closed orbit. Moreover, it can be identified with $G(\R)/P(\R)$ for some parabolic subgroup $P$ (see remark after \cite[Theorem 5.9]{KW65}).
    Then \cite[Lemma 8.5]{Mos73} gives that $\overline{\Gamma y} = \sigma(\Omega)$, and the proposition follows from $\overline{\Gamma y} \subset \overline{\Gamma x}$.
\end{proof}

\section{Proof of Theorem \ref{thm:ZariskiDenseIntersection}}\label{sec:proofDenseImage}
Now that we have Propositions \ref{prop:LimitBehavior} and \ref{prop:ShilovBoundaryLimit}, the proof of Theorem \ref{thm:ZariskiDenseIntersection} proceeds, \textit{mutatis mutandis}, like the $j$-function case found in \cite{Ete21}.

\begin{thm}[{Rouch\'e's Theorem \cite[Theorem 2.5]{Yuz83}}]
	Let $D \subset \C^n$ be a bounded domain whose boundary $\bdy D$ is piecewise smooth and let $f, g: \bar{D} \to \C^n$ be two continuous maps, whose restriction to $D$ are holomorphic and whose zeroes are isolated. 
    If at each point $\textbf{x}$ in $\bdy D$, $\|f(\textbf{x})\| > \|g(\textbf{x})\|$, then $f$ and $f + g$ have the same number of zeroes in $D$, counting multiplicities.
\end{thm}

\begin{prop}\label{prop:InfinitelyManySolutions}
	Let $\Gamma \subset G(\Q)_+$ be an arithmetic subgroup and let $q: \Omega \to S$ be the quotient map for $\Gamma$, let $U \subset \C^N$ be a Euclidean open set such that $U \cap \sigma(\Omega) \neq \emptyset$, and let $p: U \to S$ be a holomorphic map. 
 Then, the equation $q(Z) = p(Z)$ has infinitely many solutions with $Z \in U \cap \Omega$. 
 Moreover, the closure of the set of solutions contains $U \cap \sigma(\Omega)$.
\end{prop}
\begin{proof}
	Fix some $Z_0 \in U \cap \sigma(\Omega)$. 
 The quotient map $q$ is surjective so choose $Z_1 \in \Omega$ such that $q(Z_1) = p(Z_0)$. 
 There exists a small Euclidean closed ball $B \subset \Omega$ around $Z_1$ such that $q(Z) \neq q(Z_1)$ for all $Z \in B \bs \{Z_1\}$ because $q$ is locally a diffeomorphism (since $\Gamma$ is torsion-free, by convention). 
 Moreover, since $\Gamma$ is a discrete subgroup, we may shrink $B$ so that $gB$ is disjoint from $B$ for all $g \in \Gamma$ unless $gZ_1 = Z_1$. 
 By Proposition \ref{prop:ShilovBoundaryLimit}, there exists a sequence $\{\gamma_{k}\}_k \in \Gamma$ such that $\|\gamma_{k} Z_1 - Z_0 \| \to 0$. 
 By taking a subsequence, we may assume that each $\gamma_{k} Z_1$, and hence $\gamma_k B$, is disjoint.
	
	Define $\delta := \min_{Z \in \bdy B} \|q(Z) - q(Z_1)\| > 0$.
 Using the continuity of $p$, choose a Euclidean open neighborhood $W \subset U$ of $Z_0$ satisfying
	\[Z \in W \implies \|p(Z) - p(Z_0)\| < \delta/2.\]
	The supremum $\sup_{Z \in \bdy B} b_D(Z, Z_1)$ is finite because $\bdy B$ is a compact set and hence $B$ lies within a ball of finite Bergman radius.
 Since the centers $\gamma_k Z_1$ tend towards $Z_0$ on the boundary of $\Omega$, Proposition \ref{prop:LimitBehavior} implies the translates $\gamma_{k} B$ also tend uniformly to $Z_0$, and so there exists some $N$ such that $\gamma_{k} B \subset W$ for all $k > N$.
 The function $q$ is invariant under $\Gamma$ action and thus for all $Z \in \bdy (\gamma_{k} B) \subset W$
	\[\|q(Z) - p(Z_0)\| = \|q\left(\gamma_{k}^{-1}Z\right) - q(Z_1)\| \ge \delta > \|p(Z_0) - p(Z)\|.\]
	The function $q(Z) - p(Z_0)$ has an isolated zero in $\gamma_k B$ at $Z = \gamma_k Z_1$ since $q(\gamma_k Z_1) - p(Z_0) = q(Z_1) - q(Z_1) = 0$. 
 The functions $q$ and $p$ are holomorphic on $\gamma_kB$ so Rouch\'e's Theorem applied to $q(Z) - p(Z_0)$ and $p(Z_0) - p(Z)$ says their sum $q(Z) - p(Z)$ also has a zero in $\gamma_{k} B$. 
 This holds for every $k > N$, giving infinitely many solutions to the system of equations $f(Z) = p(Z)$, one in each $\gamma_{k} B$, converging to $Z_0$.
 Moreover, the $\gamma_k$ were refined so that the $\gamma_k B$ are disjoint, meaning that the solutions are all distinct.
 Our initial choice of point $Z_0 \in U \cap \sigma(\Omega)$ was arbitrary and hence all of $U \cap \sigma(\Omega)$ must lie in the closure of the set of solutions.
\end{proof}

\begin{thm}[Inverse Function Theorem]
	Let $V$ be a complex manifold of dimension $n$ and $F: V \to \C^n$ a holomorphic function. 
 If $x \in V$ is a point such that the Jacobian at $x$ has rank $n$, then there is an open neighborhood $W$ of $x$ and a holomorphic inverse $G: F(W) \to W$ such that $F \circ G = \mathrm{id}_{F(W)}$ and $G \circ F = \mathrm{id}_W$.
\end{thm}

We will use the Inverse Function Theorem in conjunction with Proposition \ref{prop:InfinitelyManySolutions} to prove Theorem \ref{thm:ZariskiDenseIntersection}.

\begin{proof}[Proof of Theorem \ref{thm:ZariskiDenseIntersection}]
First, we restrict ourselves to the case when $\dim V=N$, in which $\pi: V \to \C^N$ is a quasi-finite map with a Zariski dense image. 
The projection $\pi$ is a regular map and Chevalley's theorem implies the image of $\pi$ is a constructible set: a finite union of intersections of Zariski open and closed sets. 
Since the image is Zariski dense, we may make a further reduction to the case when the image of $\pi$ is a Zariski open set of $\C^N$.
	
Let $\mathrm{Reg}(V)$ denote the smooth locus of $V$, and let $V'\subseteq\mathrm{Reg}(V)$ denote the non-empty Zariski open subset	such that the restriction of $\pi$ to $V'$ is a local biholomorphism.
 Let $U:=\pi(V')$.
 Given $x\in V'$ and $z\in U$ satisfying $\pi(x)=z$, there exists a small Euclidean neighborhood $W$ of $x$ in $\C^N\times S$ such that $\phi:=\pi\upharpoonright_{V'\cap W}:V'\cap W\to\C^N$ is a biholomorphism onto an open subset $\pi(W)\subset\C^N$. 
 Hence, there exists an inverse holomorphic map $\psi:\pi(W)\to V'\cap W$.
 We also know that $U \cap \sigma(\Omega) \neq \emptyset$ because the latter is Zariski dense by Proposition \ref{prop:ShilovDense}. 

 Now fix $x\in V'$ with $z=\pi(x) \in U \cap \sigma(\Omega)$. 
 We can compose $\psi$ with the projection down to $S$ to obtain a holomorphic map $\pi(W) \to S$.
 Proposition \ref{prop:InfinitelyManySolutions} gives infinitely many solutions which are points in $E_q \cap V' \cap W$, and the closure of $\pi(E_q \cap V)$ contains $U \cap \sigma(\Omega)$.
 Since $U\cap\sigma(\Omega)$ is Zariski dense in $\C^N$, we conclude that $V\cap E_q$ is Zariski dense in $V$.
Indeed, if $V \cap E_q$ were not Zariski dense in $V$, then its Zariski closure would have dimension smaller than $\dim V = N$ and so the Zariski closure of $\pi(E_q \cap V)$ would have dimension smaller than $N$, a contradiction. 
	
For the general case when $\dim V > N$, suppose for the sake of argument by contradiction that $\overline{V \cap E_q}^{Zar} = W \subsetneq V$ is not Zariski dense in $V$.
 Let $V_1 = V \cap \C^N \times H$, where $H \subset S$ is an intersection of $\dim V - N$ hyperplanes chosen generically so that $V_1$ is irreducible, free, with dominant projection onto $\C^N$, and $V_1 \not \subset W$.
 Then $\dim V_1 = N$ and the above proof shows that $V_1 \cap E_q$ is Zariski dense in $V_1$. 
 But since $V_1 \cap W$ is a proper subvariety of $V_1$, there are elements of $V_1 \cap E_q$, and hence $V \cap E_q$ not lying in $W$.
 This contradicts the definition of $W$ and hence $V \cap E_q$ is Zariski dense in $V$.
\end{proof}

So far, we have only considered pure Shimura varities, but using the results of \cite{aslanyan-kirby-mantova} we can get some results for those mixed Shimura varieties which are families of abelian varieties. 
We refer the reader to \cite[Construction 2.9 and Example 2.12]{pink} for the details of the constructions.
Suppose $(G,\Omega)$ is the connected Shimura datum associated with a moduli space $p:\Omega\to S = \Gamma\bs\Omega$ of principally polarized abelian varieties of dimension $g$ (with a fixed level structure). Then, we can embed $G$ into $\mathrm{GSp}_{2g}$.

Let $V:= \Q^{2g}$ and $V_{\R}:=V\otimes_{\mathbb{Q}}\R$. 
Interpreting $V$ as a rational representation of $G$, we see that the semidirect product $G\ltimes\mathbb{G}_{a,\Q}^{2g}$ is a linear algebraic group. 
The product space $\Omega\times V_{\R}$ has an action of $G(\R)^{+}\ltimes\mathbb{G}_{a,\Q}^{2g}(\R)$ given by
\[(g,v)\cdot(x,v'):=(gx, gv'+v),\]
and to evidence this fact we use the notation $\Omega\ltimes V_{\R}$. 
Just like we can realise $\Omega$ as a subset of $\mathbb{C}^{N}$ for an appropriate $N$, so can $\Omega\ltimes V_{\R}$ be realised as an open semi-algebraic subset of a complex algebraic variety $\left(\Omega\ltimes V_{\R}\right)^{\vee}$, see \cite[\textsection 2.5]{gao}. 

Given a congruence subgroup $\Gamma$ of $G(\Q)^{+}$ and a $\Gamma$-invariant $\mathbb{Z}$-lattice $\Gamma_{V}$ of $V$ of rank $g$, the quotient $q:\Omega\ltimes V_{\R}\to B:=(\Gamma\ltimes\Gamma_{V})\bs\left(\Omega\ltimes V_{\R}\right)$ has the structure of an algebraic variety.
Let $E_{q}\subset \left(\Omega\ltimes V_{\R}\right)\times B$ be the graph of $q$.
The variety $B$ comes equipped with a morphism of Shimura varieties $\alpha:B\to S$ so that every fiber of $\alpha$ is an abelian variety of dimension $g$.
We then get the commutative diagram:
\begin{equation*}
    \begin{CD}
    \Omega\ltimes V_{\R} @>q>> B\\
    @V\pi VV @VV\alpha V\\
    \Omega @>p>> S
    \end{CD}
\end{equation*}
From this we also get a map $(\Omega\ltimes V_{\R})\times B\to\Omega\times S$ given by $((x,v),b)\mapsto (x,\alpha(b))$. 
    Let $W\subset \left(\Omega\ltimes V_{\R}\right)^{\vee}\times B$ be an algebraic variety, and suppose that there is a point of the form $(z,p(z))$ in the image of $W$ in $\Omega^{\vee}\times S$. 
The set $W_z:=\{((z,v),b)\in W : \alpha(b)=p(z)\}$ is the fiber in $W$ over $(z,p(z))$ and it determines a subvariety of $\C^g\times A_{p(z)}$, where $A_{p(z)}$ denote the fiber in $B$ over $p(z)$, so $A_{p(z)}$ is an abelian variety of dimension $g$. 
By a small abuse of notation, we will still denote the subvariety of $\C^g\times A_{p(z)}$ by $W_z$. 

\begin{cor}
\label{cor:mixed}
    With the notation as above, suppose that $W\subset \left(\Omega\ltimes V_{\R}\right)^{\vee}\times B$ is an irreducible algebraic variety such that the projection $W\to \Omega^{\vee}$ is Zariski dense.
    If the fibers in $W$ over points of the form $(z,p(z))\in\Omega^{\vee}\times S$ generically have a dominant projection to $\C^g$ (that is, outside of some proper Zariski closed subset of the image of $W$ in $\Omega^\vee\times S$), then $W\cap E_{q}$ is Zariski dense in $W$. 
\end{cor}
\begin{proof}
    Let $W'$ be the intersection of the image of $W$ in $\Omega^\vee\times S$ with $\C^N\times S$. 
    By hypothesis, $W'$ has a dominant projection to $\C^N$, so by Theorem \ref{thm:ZariskiDenseIntersection} we get that the intersection $W'\cap E_p$ is Zariski dense in $W'$. 

    We now remark that by intersecting with generic hyperplanes, we may assume that $\dim W = N+g$ while preserving the geometric conditions stated in the corollary (see \cite[Lemma 4.30]{aslanyan:adequate} and  \cite[Proposition 2.33]{kirby-semiab}). 
    Indeed, we may first intersect with generic hyperplanes until the dimension of the generic fibers of $W\to\Omega^{\vee}\times S$ is $g$, while still keeping the domination of $W'$ to $\C^N$ as well as the domination to $\C^g$ of the fibers above points of the form $(z,p(z))$.
    We can now intersect the image of $W$ in $\Omega^{\vee}\times S$ with generic hyperplanes until the dimension is $N$, and then lift these conditions to $W$. 
    
    Given $(z,p(z))\in W'\cap E_p$, we get a corresponding variety $W_z\subseteq\C^g\times A_{p(z)}$. 
    By hypothesis, outside of some (possibly empty) proper Zariski closed subset of $W'$, the points $(z,p(z))$ are such that the projection of $W_z$ to $\C^g$ is also dominant, so by \cite[Theorem 1.4]{aslanyan-kirby-mantova}, $W_z$ has a Zariski dense subset of points in the graph of $\exp:\C^{g}\to A_{p(z)}$. 
    These points can be combined with $(z,p(z))$ to produce points in $W\cap E_q$, which all together have a Zariski closure of dimension at least $N+g$.
    Since we have reduced to the case where $\dim W =N+g$, this proves the result. 
\end{proof}

One can also construct mixed Shimura varieties which are families of semiabelian varieties (see \cite[Remark 2.13]{pink}), and the proof of Corollary \ref{cor:mixed} adapts easily.
However, an analogue of \cite[Theorem 1.4]{aslanyan-kirby-mantova} for semiabelian varieties is not available yet, except for the trivial split case (see \cite[\S 6]{aslanyan-kirby-mantova}).

\section{Products of Varieties}\label{sec:WeaklySpecial}
\subsection{Totally geodesic subvarieties}
\label{subsec:totallygeodesic}
In this section, we will define totally geodesic subvarieties following \cite{Ull11} and then state some results from \cite{Ull18} on these subvarieties of $\Omega$ and their image in $S$.
\begin{defn}
	Let $(G, X)$ be Shimura datum and let $\Omega$ be a connected component of $X$.
 Let $(H, X_H)$ by a sub-Shimura datum of $(G, X)$.
 For suitable choices of compact open subsets $K$, this gives a finite map $S_H \to S$ of Shimura varieties. 
 The irreducible components of the Hecke orbits of $S_H$ are called \emph{special} subvarieties of $S$.
 Let $\Omega_H$ denote an irreducible component of $X_H$ mapping to $\Omega$.
 For each decomposition $H^{\mathrm{\ad}} \cong H_1 \times H_2$, we get a corresponding decomposition of Hermitian symmetric domains $\Omega_H = \Omega_{H, 1} \times \Omega_{H, 2}$.
 For each point $y_2 \in \Omega_2$, the image of $\Omega_1 \times \{y_2\}$ in $S$ is a \emph{weakly special} subvariety of $S$.
 The subvariety $\Omega_1 \times \{y_2\}$ is called a \emph{weakly special} subvariety of $\Omega$.
\end{defn}

There is a more general form of weakly special subvarieties called \emph{totally geodesic subvarieties} of $\Omega$, which we will be using.
The definition of totally geodesic subvarieties using geodesics would be too much of a detour for us, so instead the reader may regard Proposition \ref{prop:geodesics} as their definition.
We remark that we will diverge from the terminology used by Ullmo and Yafaev in \cite{Ull18} when referring to subvarieties of $\Omega$.
What they refer to as weakly special subvarieties, we follow the terminology of \cite{Moo98} and call them totally geodesic. 
We reserve the notion of weakly special subvarieties for totally geodesic subvarieties that are bi-algebraic.
Interested readers can also see \cite[Definition 2.2.1]{Str06}.
 
\begin{prop}[{\cite[Proposition 2.3]{Ull18}}]
\label{prop:geodesics}
	A subvariety $Z \subset \Omega$ is totally geodesic if and only if there exists a semi-simple real algebraic subgroup $F \subset G_\R$ without compact factors and some $x \in \Omega$ such that $h_x: \C^\times \to G(\R)$ factors through $FZ_G(F)^\circ$ and $Z = F(\R)^+x$.
\end{prop}

\begin{exmp}
	Totally geodesic subvarieties $Z \subset \mb{H}^n$ are also called \emph{M\"obius subvarieties} and they are cut out by equations of the form $x_i = gx_j$ for some $g \in \SL_2(\R)$ or $x_i = c$ for some $c \in \mb{H}$. 
 The totally geodesic subvariety $Z = \{(\tau, g\tau, c) : \tau \in \mb{H}\} \subset \mb{H}^3$ for fixed $g \in \SL_2(\R)$ and $c \in \mb{H}$ corresponds to the real algebraic subgroup $F(\R) := \{(h, ghg^{-1}, 1) : h \in \SL_2(\R)\}$ and $Z = F(\R) \cdot (i, gi, c)$.
	
	These M\"obius subvarieties are weakly special subvarieties precisely when all the equations of the form $x_i = gx_j$ used in their definition satisfy $g \in \mathrm{GL}_2(\Q)$.
\end{exmp}

\begin{defn}
	The \emph{Mumford--Tate group} of a real algebraic subgroup $F \subset G_\R$ is the smallest $\Q$-subgroup $H = \mathrm{MT}(F)$ of $G$ such that $F \subset H_\R$.
 We say $F$ is \emph{Hodge-generic} if $\mathrm{MT}(F) = G^{\mathrm{der}}$.
\end{defn}

We recall a result from Ratner theory on the image of these totally geodesic subvarieties of $\Omega$.

\begin{thm}[{\cite[Theorem 4.4]{Ull18}}]\label{thm:weaklySpecialDense}
	Let $F = F(\R)^+$ be a semi-simple subgroup of $G(\R)^+$ without compact factors. 
 Let $H = MT(F)$ be the Mumford--Tate group of $F$. The closure of $\Gamma \bs \Gamma F$ in $\Gamma \bs G(\R)^+$ is $\Gamma \bs \Gamma H(\R)^+$.
\end{thm}

From this and the fact that for a given $x \in \Omega$, the map $\pi_x: \Gamma \bs G(\R)^+ \to \Gamma \bs \Omega$ by $\pi_x(g) = g \cdot x$ is closed, we can now describe the image of a totally geodesic subvariety of $\Omega$ in $S$.

\begin{cor}
    Let $Z \subset \Omega$ be totally geodesic so that $Z = F(\R)^+ x$ and let $H = \mathrm{MT}(F)$. 
    Then the Euclidean closure of $q(Z)$ in $S$ is $\Gamma \bs \Gamma H(\R)^+ \cdot x$.
\end{cor}

In general, this closure is a real analytic subset that need not be an algebraic variety.
However, if we take the Zariski closure of $q(Z)$, we will get a weakly special subvariety \cite[Theorem 1.2]{Ull18}.

\begin{cor}
    Let $Z \subset \Omega$ be totally geodesic so that $Z = F(\R)^+ x$ and let $H = \mathrm{MT}(F)$. 
    Then the $Z$ is free if and only if $F$ is Hodge-generic.
\end{cor}

\begin{exmp}
    When $g \in \SL_2(\R) \bs \SL_2(\Q)$, the real algebraic subgroup $F(\R) := \{(h, ghg^{-1}) : h \in \SL_2(\R)\}$ is Hodge-generic and therefore the set of $(j(\tau), j(g\tau))$ is Euclidean dense in $\C^2$.
\end{exmp}

\subsection{Proof of Theorem \ref{thm:ProductDenseImage}}\label{sec:proofProduct}
First we prove some results on what the intersection of $q(L)$ and $W$ can look like. 
The result to be formulated as Lemma \ref{lem:NoInclusion} below will rely on the following Ax--Schanuel theorem for Shimura varieties proved in \cite{Mok19} as well as a finiteness result for weakly optimal subvarieties proved in \cite{Daw18}.

\begin{thm}[Ax--Schanuel \cite{Mok19}]
\label{thm:ax-schanuel}
    Let $V \subset \C^N \times S$ be an algebraic subvariety and let $U$ be an irreducible component of the intersection $V \cap E_q$ with the graph of the quotient map $q$.
    If $\dim U > \dim V - N$, then the projection of $U$ to $S$ is contained in a proper weakly special subvariety of $S$.
\end{thm}

\begin{lem}\label{lem:NoInclusion}
    Let $W \subset S$ be a free proper subvariety.
    There exists a Zariski open subset $Z \subset W$ such that for any totally geodesic subvariety $L \subset \Omega$ such that $L\times W$ is broad, any intersection component of $q(L) \cap W$ that intersects $Z$ has the expected dimension $\dim L + \dim W - \dim S$.
\end{lem}
\begin{proof}
    For each splitting of $G^{\ad} = G_1 \times G_2$ into two normal $\Q$-subgroups, we get a splitting of Hermitian domains $\Omega = \Omega_1 \times \Omega_2$, and as in Definition \ref{defn:decomposition of Shimura Variety}, we may choose levels $K_i$ suitably to get projections $p_i \coloneqq S \to S_i$.
    The fiber-dimension theorem implies that there exist Zariski open sets $U_i \subset \overline{p_i(W)}^{\mathrm{Zar}}$ such that for any $s \in U_i$, the dimension of the fiber is as expected $\dim W_s = \dim W - \dim p_i(W)$.
    By a result of Daw and Ren (\cite[Proposition 6.3]{Daw18}), there exists a finite set $\Sigma$ of Shimura subdata of $(G, X)$ such that all weakly optimal subvarieties of $W$ are described by splittings of the Shimura subdata in $\Sigma$.
    By the discussion after Conjecture 15.4 of \cite{Daw18}, all maximal atypical intersections with $W$ are optimal and hence weakly optimal (\cite[Corollary 4.5]{Daw18}).
    Thus, all maximal atypical intersections with $W$ are described by splittings of the Shimura subdata in $\Sigma$.
    We remove $(G, X)$ from $\Sigma$ so that $\Sigma$ only contains proper Shimura subdata.
    Define
    \[Z := \bigcap_{\Omega = \Omega_1 \times \Omega_2} \left(W \cap p_1^{-1}(U_1) \cap p_2^{-1}(U_2)\right) \cap \bigcap_{S_M \in \Sigma} \left(W \cap S_M^c\right),\]
    where the first intersection is over all splittings of $G^{\ad}$.
    Since there are only finitely many ways to split $G^{\ad} = G_1 \times G_2$ into a product of $\Q$-subgroups and finitely many special subvarieties in $\Sigma$, none of which contain $W$, the intersection $Z$ is a Zariski open subset of $W$.
    
    Now let $U$ be an intersection component of $q(L) \cap W$ so that $U \cap Z \neq \emptyset$ and suppose that $\dim U > \dim L + \dim W - \dim S$.
    Then Ax--Schanuel (Theorem \ref{thm:ax-schanuel}) says that $U$ is contained in a proper weakly special subvariety of $S$.
    Thus $U$ lies in a maximal atypical intersection of $W$ with a splitting of a Shimura subdatum in $\Sigma \cup \{(G, X)\}$.
    Since $U$ intersects $Z$, the proper weakly special subvariety containing $U$ must be obtained from a splitting of $G^{\ad} = G_1 \times G_2$.
    This gives rise to a splitting of Hermitian symmetric domains $\Omega = \Omega_1 \times \Omega_2$ and a finite map $p \colon S \to S_1 \times S_2$ so that $U$ being contained in a weakly special subvariety means there exists $x \in \Omega_2$ such that $p(U) \subset q(\Omega_1 \times \{x\})$. 
    We may choose this splitting so that $p_1(U)$ is free inside $S_1$.
    
    The projection $p_2(U)$ is constant on $S_2$, so $U$ is actually in the intersection of the fibers $q(L_x) \cap p(W)_{q(x)} \subset S_1$.
    But by construction, $U$ is not contained in a proper weakly special subvariety of $S_1$ and hence $\dim U = \dim L_x + \dim p(W)_{q(x)} - \dim S_1$.
    By the fiber-dimension theorem, we can write the dimension of the fiber of $p(W)$ as $\dim p(W)_{q(x)} = \dim p(W) - \dim p_2(W)$ and $\dim p_2(L) + \dim p_2(W) \ge \dim S_2$ by the broadness condition. 
    The map $p$ is finite so $\dim W = \dim p(W)$.
    Combining this with $\dim U > \dim L + \dim W - \dim S$ gives
    \[\dim L < \dim L_x + \dim p_2(L).\]
    But, this is impossible since $L$ is totally geodesic and equals the orbit of a point under the group action, meaning that all of its fibers are of the same dimension.
    Thus $U$ must have expected dimension.
\end{proof}

We are now ready to prove the theorem.

\begin{proof}[Proof of Theorem \ref{thm:ProductDenseImage}]
We first show that $q(L) \cap W \neq \emptyset$.
There are only finitely many ways to split $G^{\ad} = G_1 \times G_2$ and hence $\Omega = \Omega_1 \times \Omega_2$. 
Furthermore, all the weakly special subvarieties of $S$ can be placed into countably many algebraic families of subvarieties of $S$. 
So we can cut $W$ down with generic hyperplanes to maintain that $W$ is still free and that $L\times W$ is broad, and furthermore $\dim L = \text{codim} W = d$.
	
Choose a real semi-simple subgroup $F = F(\R)^+ \subset G(\R)^+$ and point $x \in \Omega$ such that $L = Fx$. 
For each smooth point $w \in Z \subset W$, where $Z$ is as in Lemma \ref{lem:NoInclusion}, choose some $g \in G^{\mathrm{der}}(\R)^+$ such that $w = q(gx)$.
We will show that in any open neighborhood of $w$, there exists a point in the intersection $W \cap q(L)$.
If $g \in F$ then $w \in q(L)$ and we are done, so assume that $g \not \in F$. 
The group $F$ is Hodge-generic meaning $\mathrm{MT}(F) = G^{\mathrm{der}}$ so by Theorem \ref{thm:weaklySpecialDense}, there exists a sequence $\{f_i\}_{i\in\N} \subset F(\R)^+$ and a sequence $\{\gamma_i\}_{i\in\N} \subset \Gamma$ such that $\gamma_i f_i \to g$ as $i\to\infty$. 
 Since $w$ is a smooth point of $W$, in a small neighborhood $V$ of $w$, there exist functions $p_1, \dots, p_d: S \to \C$ such that $W$ is cut out by the $p_i$. 
 Define $P: L \to \C^d$ as the function $P(z) = (p_1(q(gz)), \dots, p_d(q(gz)))$ and $P_i: L \to \C^d$ by $P_i(z) = (p_1(q(\gamma_if_i z)), \dots, p_d(q(\gamma_if_i z)))$.
Hence we have $P_i \to P$ locally uniformly as $i\to\infty$, and $P(x) = 0$.
	
Since $\dim L = \text{codim} W$, $w\in Z$ and $gL\times W$ is broad, the components of $q(gL)\cap W$ that intersect $Z$ have dimension 0 by Lemma \ref{lem:NoInclusion}, and so $x$ is an isolated zero of $P$ and there exists a Euclidean open neighborhood $U \subset L$ containing $x$ such that $P(y) \neq 0$ for all $y \in \overline{U} \bs \{x\}$.
For each $N \in \N$, let $U_N \coloneqq U \cap B(x, \frac{1}{N})$, the points of $U$ that lie within a radius $\frac{1}{N}$ ball centered at $x$.
Let $\epsilon_N = \inf_{y \in \bdy U_N} \|P(y)\| > 0$. 
The $P_i$ locally uniformly converge to $P$ and $\bdy U_N$ is compact, so there exists $i_N$ such that $\sup_{y \in \bdy U_N} \|P_{i_N}(y) - P(y)\| < \epsilon_N$.
Rouch\'e's theorem implies that there exists $x_N \in U_N$ such that $P_{i_N}(x_N) = 0$. 
This means that $q(\gamma_{i_N} f_{i_N} x_N) = q(f_{i_N} x_N) \in W$. 
Since $x_N \in L$ and $f_{i_N} \in F(\R)^+$, we have $f_{i_N}x_N \in L$ as well.
Observe that $x_N \to x$ as $N\to\infty$ and hence $\gamma_{i_N}f_{i_N}x_N \to gx$.
Thus we have a sequence $f_{i_N}x_N \in L$ such that $q(f_{i_N}x_N) \to w$ and $q(f_{i_N}x_N) \in W$, and thus any open neighborhood of $w$ contains a point in $q(L) \cap W$.
Since this is true for any smooth point of $W$ lying in $Z$, we get a Euclidean dense intersection $W \cap q(L)$ inside of $W$.

\end{proof}

\bibliographystyle{alpha}
\bibliography{references}{}

\end{document}